\documentclass{article}
\usepackage{maa-monthly}

\theoremstyle{Theorem}
\newtheorem{thm}{Theorem}
\newtheorem{cor}{Corollary}
\newtheorem{conj}{Conjecture}
\newtheorem{prop}{Proposition}
\newtheorem{lem}{Lemma}

\theoremstyle{definition}

\newcommand{\floor}[1]{\l\lfloor #1\r\rfloor}
\newcommand{\ceil}[1]{\l\lceil #1\r\rceil}
\newcommand{\half}{\frac{1}{2}}
\newcommand{\quart}{\frac{1}{4}}
\newcommand{\tr}[1]{\textrm{#1}}
\newcommand{\rec}[1]{\frac{1}{#1}}
\newcommand{\f}[2]{\frac{#1}{#2}}

\newcommand{\sig}{\sigma}

\newcommand{\Bin}{\mathrm{Bin}}
\newcommand{\lam}{\lambda}

\newcommand{\Z}{\mathbb{Z}}
\newcommand{\R}{\mathbb{R}}
\newcommand{\sub}{\subset}
\newcommand{\sm}{\setminus}
\newcommand{\al}{\alpha}
\renewcommand{\l}{\left}
\renewcommand{\r}{\right}
\newcommand{\mr}[1]{\textup{#1}}

\begin{document}
	
	\title{The Maximum Number of Appearances of a Word in a Grid}
	\markright{Number of Appearances of a Word in a Grid}
	\author{Gregory Patchell and Sam Spiro}
	
	\maketitle
	
	\begin{abstract}
		How can you fill a $3\times 3$ grid with the letters A and M so that the word ``AMM'' appears as many times as possible in the grid?  More generally, given a word $w$ of length $n$, how can you fill an $n\times n$ grid so that $w$ appears as many times as possible?  We solve this problem exactly for several families of words, and we asymptotically solve this problem in higher-dimensional grids.
	\end{abstract}
	
	\section{Introduction.}
	Consider the following problem.  You have been asked to advertise \textit{The American Mathematical Monthly} using a $3\times3$ grid.  Specifically, you are asked to fill in the squares of the grid with A's and M's to maximize the number of times the word ``AMM'' appears along a horizontal, vertical, or diagonal line; and we will also count a line if it includes the word written backwards.  
	
	Given this task, you might think for a bit and come up with a number of designs that seem to do well.  Maybe your first thought is to fill in the squares as in Figure \ref{fig:demo}, which has 5 instances of the word ``AMM.''  
	
	\begin{figure}[h]
		\centering
		\includegraphics[width=.3\textwidth]{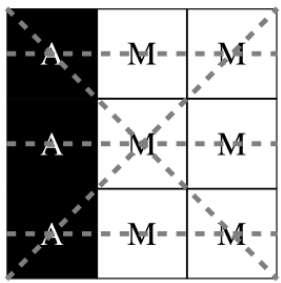}
		\caption{A 3x3 grid with 5 instances of the word AMM.}
		\label{fig:demo}
	\end{figure}
	
	However, being the mathematician that you are, you are not convinced that this is the best you can do.  You might think for a bit and come up with more complicated designs like those in Figure \ref{fig:all-3x3}, which also all contain 5 copies of the word.  Is there a design giving more than 5 copies of the word AMM?
	
	\begin{figure}[h]
		\centering
		\includegraphics[width=.6\textwidth]{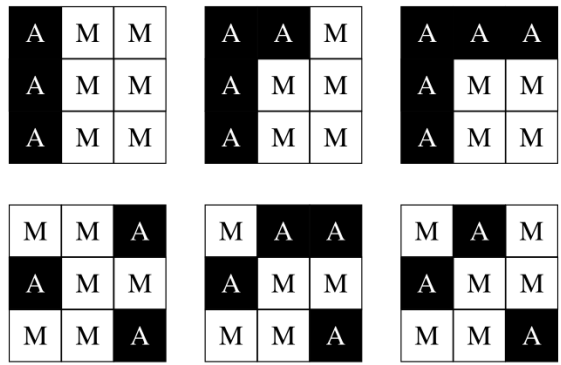}
		
		\caption{Several grids with 5 instances of AMM.}
		\label{fig:all-3x3}
	\end{figure}
	
	Of course, once we successfully solve this problem, other organizations are going to want to hire us for similar advertising jobs. To prepare for these new jobs, we let $[n]=\{1,2,\ldots,n\}$ and $[n]^2=\{(i,j):i,j\in [n]\}$.  We define an \textit{$n$-grid} $G$ to be a function from $[n]^2$ to a set of letters, and we will simply call this a \textit{grid} whenever $n$ is understood.  Given a word $w=w_1\cdots w_n$ and an $n$-grid $G$, we say that the $i$th row of $G$ \textit{contains} $w$ if $G(i,j)=w_j$ for all $1\le j\le n$, or if $G(i,j)=w_{n-j+1}$ for all $1\le j\le n$.  We similarly define what it means for the $i$th column of $G$ to contain $w$, as well as for the diagonals of $G$ to contain $w$.  We let $f(w,G)$ be the total number of rows, columns, and diagonals of $G$ containing $w$, and we define $f(w)=\max_G f(w,G)$, where the maximum ranges over all $n$-grids $G$.  More generally, if $\mathcal{W}$ is a set of words of length $n$, we define $f(\mathcal{W},G)$ to be the number of rows, columns, and diagonals of $G$ containing some $w\in \mathcal{W}$, and we similarly define $f(\mathcal{W})$.
	
	We note that determining which grids $G$ satisfy $f(\mathcal{W},G)=f(\mathcal{W})$ can be seen as a generalization of a classical problem concerning Latin squares.  In our language, a \textit{Latin square} is an $n$-grid $G$ such that every row and column of $G$ contains a permutation of the elements of $[n]$; see Figure~\ref{fig:Latin} for an example.

	\begin{figure}[h]
		\centering
		\includegraphics[width=.3\textwidth]{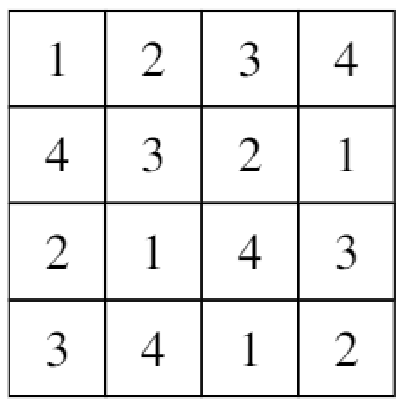}
		\caption{A (diagonal) Latin square.}
		\label{fig:Latin}
	\end{figure}

	If both diagonals of $G$ also contain a permutation on $[n]$ (as is the case in Figure~\ref{fig:Latin}), we call $G$ a \textit{diagonal Latin square}.   Observe that $G$ being a diagonal Latin square is equivalent to having $f(\mathcal{P},G)=2n+2$ where $\mathcal{P}$ is the set of all permutations of $[n]$.  Thus grids $G$ with $f(\mathcal{W},G)=f(\mathcal{W})$ where $\mathcal{W}$ is a set of words can be viewed as a generalization of diagonal Latin squares.  Much more can be said about Latin squares; see, for example, \cite{DK, DM, E, MW, U, ZK}.
	
	In this paper we focus primarily on the problem of determining $f(w)$ for a single word $w$, and even this humble task seems to be difficult to solve in general. Nevertheless, we are able to compute $f(w)$ exactly for several natural choices of words $w$.  
	
	As a point of reference, for all words $w$ of length $n$, we will show that  \[n+2\le f(w)\le 2n+2.\]  This lower bound tends to be closer to the truth when $w$ has no letter appearing many times, or when $w$ is symmetric.
	
	\begin{thm}\label{T-Few}
		If $w$ is a word of length $n\ge 2$ such that each letter occurs at most $n/4$ times, then \[f(w)=n+2.\]
		Moreover, for infinitely many $n\ge 2$ there exists a word $w$ of length $n$ such that every letter occurs at most $1+n/4$ times and such that $f(w)>n+2$.
	\end{thm}
	\begin{thm}\label{T-Sym}
		If $w$ is a word of length $n\ge 2$ such that $w_i=w_{n-i+1}$ for all $i$, then \[f(w)=\max\{n,2k\}+2,\] where $k$ is the maximum number of times any letter appears in $w$.
	\end{thm}
	
	Conversely, when a word is ``anti-symmetric,'' the upper bound tends to be closer to the truth.
	\begin{prop}\label{P-AntiSym}
		If $w$ is a word of length $n\ge 2$ using only two letters, and if $w_i\ne w_{n-i+1}$ for all $i$, then \[f(w)=2n.\]
	\end{prop}
	
	This together with Theorem~\ref{T-Sym} gives the following peculiar result, which demonstrates how sensitive $f(w)$ can be to the symmetries of $w$.
	
	\begin{cor}
		Let $w$ be the word of length $n\ge 2$ defined by $w=\mr{AMAM}\cdots$.  Then \[f(w)=\begin{cases}
			n+2 & n\tr{ odd},\\ 
			2n & n\tr{ even}.
		\end{cases}\]
	\end{cor}
	
	And of course, we also compute $f(\mr{AMM})$.  More generally, we prove the following, where by $\mr{A}^k \mr{M}^{n-k}$ we mean the word which starts with $k$ copies of the letter $\mr{A}$ followed by $n-k$ copies of the letter $\mr{M}$.
	
	\begin{thm}\label{T-FFT}
		Let $w=\mr{A}^{k}\mr{M}^{n-k}$ with $1\le k\le n/2$ and $n\ge 2$.  Then \[f(w)=\max\{2(n-k)+1,4k\}.\]
	\end{thm}
	In particular, taking $n=3,k=1$ shows that indeed $f(\mr{AMM})=5$.
	
	We consider a generalization of this problem to higher-dimensional grids, which might be useful, for example, if the AMM wanted to expand its readership to 4-dimensional beings. We formally make our definitions in Section~\ref{S-d}.  Informally, we let $f(w,d)$ be the maximum number of lines a $d$-dimensional $n$-grid can have which contain $w$.  We recall the notation $f\sim g$ for functions $f,g$ to mean that $\lim_{d\to \infty} \f{f(d)}{g(d)}=1$.  With this in mind, we have the following.
	
	\begin{thm}\label{T-HighD}
		Let $w$ be a word of length $n\ge 2$.  If $w_i=w_{n-i+1}$ for all $i$, then \[f(w,d)\sim \half (n+2)^d,\]  and otherwise \[f(w,d)\sim \quart (n+2)^d.\]
	\end{thm}
	As a point of reference, it is well-known that $d$-dimensional $n$-grid contains roughly $\half (n+2)^d$ total lines \cite{BPV}. Thus, despite the fact that $f(w,2)$ can be quite small for symmetric words, in higher-dimensional grids we can make it so that almost every line contains a copy of $w$. 
	
	Determining $f(w,d)$ exactly for all $d$ seems very difficult in general.  Remarkably, we can do this for certain anti-symmetric words, namely those mentioned in Proposition~\ref{P-AntiSym}.
	
	\begin{thm}\label{T-AntiSym}
		If $w$ is a word of length $n\ge 2$ using only two letters, and such that $w_i\ne w_{n-i+1}$ for all $i$, then for all $d$, \[f(w,d)=\quart((n+2)^d-(n-2)^d).\]
	\end{thm}
	
	Another natural generalization of this problem is to bound the maximum number of times a word $w$ of length $k$ can appear in an $(n,d)$-grid with $k\le n$, which we denote by $f(w,n,d)$.  Even in the $d=2$ case and for simple words we do not have asymptotic formulas for $f(w,n,d)$, though we are able to obtain some reasonable asymptotic bounds.  We briefly discuss these results in Section~\ref{sec:k}.
	\section{Words on a Plane.}
	\subsection{General lower bounds.}
	From now on, whenever we have a word $w$, we assume it is of length $n$ and that $n\ge 2$. We first consider basic constructions of grids to give lower bounds on $f(w)$.  The simplest such result is the following. 
	\begin{lem}\label{L-TrivLow}
		For all $w$, \[f(w)\ge n+2.\]
	\end{lem}
	
	For example, Figure \ref{fig:lem1} demonstrates how we can achieve 5 lines for $w = \mr{ABC}$. 
	
	\begin{figure}[h]
		\centering
		\includegraphics[width=.25\textwidth]{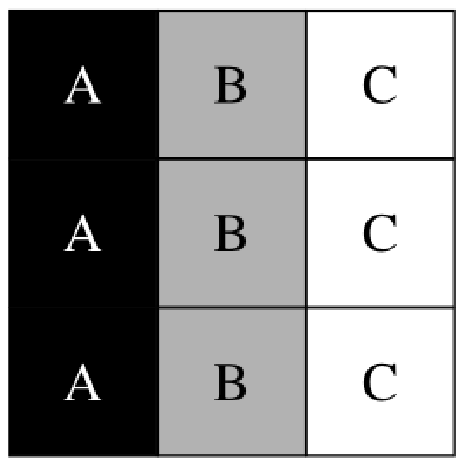}
		\caption{A $3\times3$ grid with 5 instances of the word ABC.}
		\label{fig:lem1}
	\end{figure}

	\begin{proof}
		Let $G$ be the grid that has $G(i,j)=w_j$ for all $i,j$.  Then every row and diagonal of $G$ contains $w$, so $f(w)\ge f(w,G)\ge n+2$.
	\end{proof}

	We can improve upon this bound by exploiting certain structures of the given word $w$. For example, we have the following.
	
	\begin{lem}\label{L-ManyLow}
		Let $w$ be a word such that some letter $\mr{A}$ appears $k$ times.  Then \[f(w)\ge 2k+1.\]
		Further, if $w_i=w_{n-i+1}$ whenever $w_i=A$, then \[f(w)\ge 2k+2.\]
	\end{lem} 
	
	Figure \ref{fig:lem2}(a) gives an example of the construction of Lemma \ref{L-ManyLow} with the word BAACA and Figure \ref{fig:lem2}(b) does the same with the word ABACA. Note that the latter grid achieves its target word ABACA in both of its diagonals while the former does not.
	
	\begin{figure}[h]
		\centering
		\includegraphics[width=.7\textwidth]{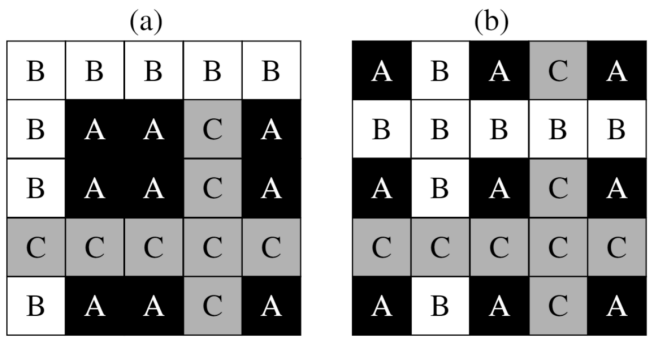}
		
		\caption{(a) A grid with 7 copies of BAACA. (b) A grid with 8 copies of ABACA.}
		\label{fig:lem2}
	\end{figure}

	\begin{proof}
		Let $I$ be the set of indices with $w_i=\mr{A}$ (this corresponds to $I=\{2,3,5\}$ in Figure~\ref{fig:lem2}(a) and $I=\{1,3,5\}$ in Figure~\ref{fig:lem2}(b)).  Define the grid $G$ by setting $G(i,j)=w_j$ whenever $i\in I$ and $G(i,j)=w_i$ otherwise.  That is, we write out $w$ along each row corresponding to $I$, and we then try to do this along each column.  By construction, the $i$th row of $G$ contains $w$ whenever $i\in I$, and also $G(i,i)=w_i$ for all $i$.  If $i\in I$, we claim that the $i$th column also contains $w,$ that is, $G(j,i)=w_j$ for all $j$.  This is immediate if $j\notin I$; otherwise $G(j,i)=w_i=\mr{A}=w_j$.  Thus in total we have that $f(w)\ge f(w,G)\ge 2|I|+1=2k+1$.
		
		If $w_i=w_{n-i+1}$ for all $i\in I$, then we claim that $G(i,n-i+1)=w_i$ for all $i$.  This is immediate if $i\notin I$, and otherwise $G(i,n-i+1)=w_{n-i+1}=\mr{A}=w_i$.  This gives an extra diagonal containing $w$, proving the second bound.
	\end{proof}

	While Lemma~\ref{L-ManyLow} gives a small gain when $w$ is somewhat symmetric, we can often do much better if the word is anti-symmetric.
	\begin{lem}\label{L-SymLow}
		Let $w$ be a word and $\mr{A, M}$ two letters used in $w$. Define \[T=\{i:w_i=\mr{A}, \ w_{n-i+1} = \mr{M}\},\] and let $t = |T|$.  Then 
		\[f(w)\ge 4t.\]
		Moreover, if $w$ only uses the letters $\mr{A}$ and $\mr{M}$, then 
		\[f(w)\ge n+t.\]
	\end{lem}
	
	Figure \ref{fig:lem3}(a) gives the $4t$ construction and Figure \ref{fig:lem3}(b) gives the $n+t$ construction for the word AMAAM. In this example, $T=\{1,4\}$ and X is an arbitrary letter.
	
	\begin{figure}[h]
		\centering
		\includegraphics[width=.7\textwidth]{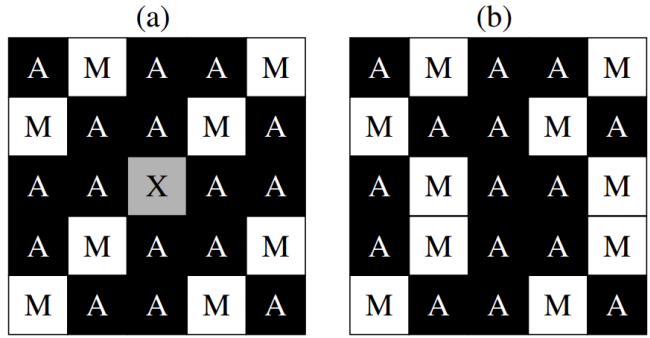}
		\caption{(a) A grid with $4\cdot 2$ copies of AMAAM. (b) A grid with $5+2$ copies of AMAAM.}
		\label{fig:lem3}
		
	\end{figure}
	
	\begin{proof}
		Define $S=\{i:w_{n-i+1}=\mr{A},\ w_i = \mr{M}\}$, and note that by definition this set is disjoint from $T$. Furthermore, $|S| = t$ since $i\in T$ if and only if $n-i+1\in S$. Let $G$ be the grid with
		\[G(i,j)=\begin{cases}w_j & i\in T,\\ 
			w_i & j\in T,\\ 
			w_{n-j+1} & i\in S,\\ 
			w_{n-i+1} & j\in S,\\ 
			X & \mathrm{otherwise},
		\end{cases}
		\]
		where X is an arbitrary letter.  It is not difficult to see that this grid is well-defined when either $i,j\in T$ or $i,j\in S$ (since $w_i=w_j$ and $w_{n-i+1}=w_{n-j+1}$ in these cases).  If, say, $i\in T$ and $j\in S$, then $w_{n-i+1}=\mr{M}=w_j$, so it is well-defined in this case as well.  By construction, every row and column corresponding to $T$ and $S$ contains $w$, so $f(w)\ge f(w,G)\ge 4t$.
		
		If $w$ consists of only $\mr{A}$'s and $\mr{M}$'s, then we define the grid $G$ by $G(i,j)=w_j$ if $w_i=\mr{A}$ and $G(i,j)=w_{n-j+1}$ otherwise.  By construction, all of the $n$ rows of $G$ contain $w$ written forwards or backwards.  We also claim for $i\in T$ that the $i$th column of $G$ contains $w$.  Indeed, for such $i$, if $w_j=\mr{A}$ then $G(j,i)=w_i = \mr{A}=w_j$, and otherwise $G(j,i)=w_{n-i+1}=\mr{M}=w_j$.  Thus for every $i\in T$, $G(j,i)=w_j$, so each of these columns contains $w$ and we have $f(w)\ge f(w,G)\ge n+t$.
	\end{proof}
	
	\subsection{General upper bounds.}
	Because there are $2n+2$ lines in an $n$-grid, we have $f(w)\le 2n+2$ for all $w$.  This bound is sharp when $w$ is the constant word $\mr{A}^n$.  We can slightly improve this bound for nonsymmetric words.
	
	\begin{lem}\label{L-SharpTot2}
		Let $w$ be a word such that $w_i\ne w_{n-i+1}$ for some $i$.  Then \[f(w)\le 2n.\]
	\end{lem}
	\begin{proof}
		Let $G$ be an $n$-grid and define \[Q=\{(i,i),(i,n-i+1),(n-i+1,i),(n-i+1,n-i+1)\},\] \[Q^+=\{s\in Q:G(s)=w_i\},\hspace{3em} Q^-=\{s\in Q:G(s)=w_{n-i+1}\}.\]   Given two distinct $q,q'\in Q$, let $\ell_{q,q'}$ be the unique line of $G$ containing $q$ and $q'$.  Observe that this line can contain $w$ only if either $q\in Q^+$ and $q'\in Q^-$ or the other way around.  Thus out of the 6 lines $\ell_{q,q'}$, the number of these that can contain $w$ is at most \[|Q^+||Q^-|\le \l(\half|Q|\r)^2=4.\] 
		In particular, at least two of the $2n+2$ lines of $G$ cannot contain $w$, giving the result.
	\end{proof}
	
	Next we consider an upper bound that works well when there are not too many copies of a given letter.

	\begin{lem}\label{L-FewUp}
		Let $w$ be a word such that each letter appears at most $k$ times in $w$.   Then \[f(w)\le \max \{4k,n\}+2.\]
	\end{lem}
	\begin{proof}
		Let $G$ be a grid and assume $f(w,G)>4k+2$.  Let $R_1$ denote the set of $i\in [n]$ such that $G(i,j)=w_j$ for all $j$, $R_2$ the set of $i\in [n]$ such that $G(i,j)=w_{n-j+1}$ for all $j$, and similarly define $C_1$ and $C_2$ for the columns.  Note that $f(w,G)\le |R_1|+|R_2|+|C_1|+|C_2|+2$, and because $f(w,G)>4k+2$, one of these sets must have size at least $k+1$.  Without loss of generality we can assume that $R_1$ has this property.
		
		We claim that no column of $G$ contains $w$.  Indeed, in the $i$th column we have at least $k+1$ different $j\in R_1$ such that $G(j,i)=w_i$.  However, each letter of $w$ appears at most $k$ times, so this column cannot contain $w$.  Thus for such a grid we have $f(w,G)\le n+2$.
	\end{proof}
	
	A similar proof gives strong bounds when a word is very symmetric.
	\begin{lem}\label{L-SymUp}
		Let $w$ be a word such that each letter appears at most $k$ times in $w$ and such that there are $s$ indices $i$ with $w_i=w_{n-i+1}$.  Then \[f(w)\le \max\{n+2k-s,n\}+2.\]
	\end{lem}
	\begin{proof}
		Let $G$ be a grid and define $R_1,R_2,C_1,C_2$ as in the proof of Lemma~\ref{L-FewUp}.  We can assume without loss of generality that $|R_1|+|R_2|\ge |C_1|+|C_2|$.  If $|R_1|$ or $|R_2|$ is strictly larger than $k$, then no column of $G$ can contain $w$ (since each column would contain more than $k$ copies of some letter), so $f(w)\le n+2$ in this case.  Thus we can assume $|R_1|,|R_2|\le k$.
		
		If $|R_1|+|R_2|>k$ and $i$ is such that $w_i=w_{n-i+1}$, then the $i$th column of $G$ contains more than $k$ copies of the letter $w_i$ and hence cannot contain $w$.  Because there are $s$ such $i$, we have \[f(w,G)\le |R_1|+|R_2|+|C_1|+|C_2|+2\le 2k+n-s+2.\]
		
		Finally, if $|R_1|+|R_2|\le k$ then also $|C_1|+|C_2|\le k$ by assumption, and we have $f(w,G)\le 2k+2\le 2k+n-s+2$, where this last step used $s\le n$.
	\end{proof}
	
	We claim without proof that the argument used to prove Lemma~\ref{L-FewUp} generalizes as follows, which could be useful for bounding $f(w)$ for other kinds of words.
	\begin{prop}
		Let $w$ be a word on the letters $\mr{A}_1,\ldots,\mr{A}_m$ such that $\mr{A}_i$ appears $k_i$ times in $w$ with $k_1\ge k_2\ge \cdots\ge k_m$.  Then for all grids $G$ and all $1\le i\le m$, \[f(w)\le \max \{4k_i,n+\sum_{j<i} k_j\}+2.\]
	\end{prop}
	Note that the $i=1$ case is exactly the statement of Lemma~\ref{L-FewUp}, and for example the $i=2$ case gives better bounds for $f(\mr{A}^{n/3} \mr{M}^{2n/3})$.  A similar generalization can be made for Lemma~\ref{L-SymUp}.

	\subsection{Putting the pieces together.}
	In this section we use the lemmas we have developed to prove our main results for $f(w)$.  The first few results are immediate.

	\begin{proof}[Proof of Theorem~\ref{T-Few}]
		For the first result, the lower bound follows from Lemma~\ref{L-TrivLow} and the upper bound from Lemma~\ref{L-FewUp}.  For the second result, take $n=4k$ for any integer $k\ge 1$.  Consider the word $w$ that begins with $k+1$ copies of the letter $A$, that ends with $k+1$ copies of the letter $M$, and that has $2k-2$ distinct letters in the remaining positions.  Note that each letter in $w$ appears at most $k+1=n/4+1$ times.  By Lemma~\ref{L-SymLow}, we have \[f(w)\ge 4(k+1)>4k+2=n+2.\]
	\end{proof}
	
	\begin{proof}[Proof of Theorem~\ref{T-Sym}]
		The lower bounds follow from Lemmas~\ref{L-TrivLow} and \ref{L-ManyLow}.  The upper bound follows from Lemma~\ref{L-SymUp}.
	\end{proof}
	\begin{proof}[Proof of Proposition~\ref{P-AntiSym}]
		The lower bound follows from Lemma~\ref{L-SymLow}.  The upper bound follows from Lemma~\ref{L-SharpTot2}.
	\end{proof}
	
	Computing $f(\mr{A}^k \mr{M}^{n-k})$ will be a bit more involved. 
	
	\begin{proof}[Proof of Theorem~\ref{T-FFT}]
		The lower bounds follow from Lemmas~\ref{L-ManyLow} and \ref{L-SymLow}.  For the upper bound, we say that the $i$th row of a grid is a \textit{top row} if $i\le k$, a \textit{middle row} if $k<i<n-k$, and a \textit{bottom row} if $i\ge n-k$.  We similarly define left, middle, and right columns.  Let $G$ be a grid.  An observation that we use throughout this proof is that if $G(i,j)=G(i,j')$ with $j\le k$ and $j'\ge n-k$, then the $i$th row cannot contain $w$ (as we would need one of these entries to be A and the other to be M).  Similarly if $G(i,j)=\mr{A}$ with $k<j<n-k$, then the $i$th row cannot contain $w$. Analogous results hold with the roles of columns and rows swapped.
		
		First consider the case that both diagonals of $G$ contain $w$. By rotating $G$, we can assume $G(1,1)=G(1,n)=\mr{A}$ and that $G(n,1)=G(n,n)=\mr{M}$.  Because $G(i,i)=G(i,n-i+1)=\mr{A}$ for $1\le i\le k$, none of the $k$ top rows can contain $w$, and similarly none of the $k$ bottom rows can contain $w$.  If no row contains $w$, then $f(w,G)\le n+2$.  This equals $\max\{4k,2(n-k)+1\}$ only when $n=2,k=1$ or $n=3,k=1$, and for any other values it is strictly smaller.  Thus we can assume that some number of the $n-2k$ middle rows contain $w$, say with $G(i,j)=w_j$ for some $k<i<n-k$ and all $j$; see Figure \ref{fig:thm3-1} for an example. This implies that none of the $k$ left columns contain $w$, so in total counting rows, columns, and diagonals we have \[f(w,G)\le (n-2k)+(n-k)+2\le 2(n-k)+1,\] since $k\ge 1$.

		\begin{figure}[!htbp]
			\centering
			\includegraphics[width=.3\textwidth]{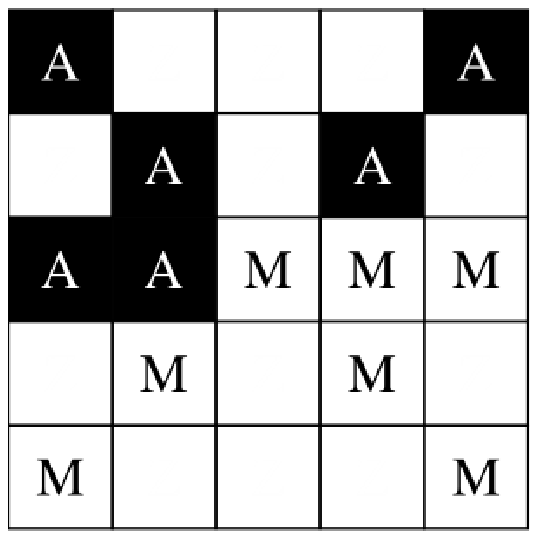}
			\caption{Both diagonals filled in as well as one middle row. Now neither of the two left columns can contain
				AAMMM.}
			\label{fig:thm3-1}
		\end{figure}

		Now assume exactly one diagonal of $G$ contains $w$.  By rotating $G$ we can assume that this is the diagonal containing $(1,1)$ and $(n,n)$ and that $G(1,1)=\mr{A},\ G(n,n)=\mr{M}$.  If some $i$th row of $G$ with $i\ge n-k$ contains $w$, then we must have $G(i,j)=\mr{A}$ for $j\le k$ since $G(i,i)=\mr{M}$ by assumption of the diagonal containing $w$. Because $G(i,j)=\mr{A}$ for all $j\le k$, we conclude that, as in Figure \ref{fig:thm3-2}, none of the left columns can contain $w$ if one of the bottom rows contains $w$.  Similarly, either $G$ contains none of the $k$ right columns or none of the $k$ top rows, so in total we have $f(w,G)\le 2(n-k)+1$.

		\begin{figure}[!htbp]
			\centering
			\includegraphics[width=.3\textwidth]{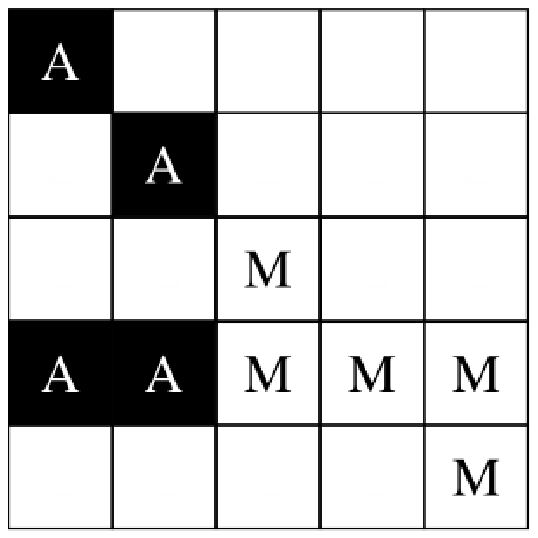}
			\caption{One diagonal filled in as well as one bottom row. Neither of the left two columns can contain AAMMM.}
			\label{fig:thm3-2}
		\end{figure}
		
		Finally we assume no diagonal of $G$ contains $w$.  If none of the middle $n-2k$ rows or columns contain $w$, then $f(w,G)\le 2n-2(n-2k)=4k$.  If there exists $k<i,i'<n-k$ such that both the $i$th row and $i'$th column contain $w$, say with $G(i,j)=G(j,i')=w_j$ for all $j$, then, as in Figure \ref{fig:thm3-3}, none of the $k$ top rows or left columns can contain $w$, so $f(w,G)\le 2n-2k$.  If none of the $n-2k$ middle columns contain $w$ and if there is some $k<i<n-k$ with $G(i,j)=w_j$ for all $j$, then none of the $k$ left columns contain $w$ and we have $f(w,G)\le 2n-(n-2k)-k=n+k$.  Note that $n+k< 4k$ when $k>n/3$ and that $n+k<2(n-k)+1$ for $k\le n/3$, so $f(w,G)$ is strictly smaller than the stated upper bound in this case.
		
		\begin{figure}[!htbp]
			\centering
			\includegraphics[width=.3\textwidth]{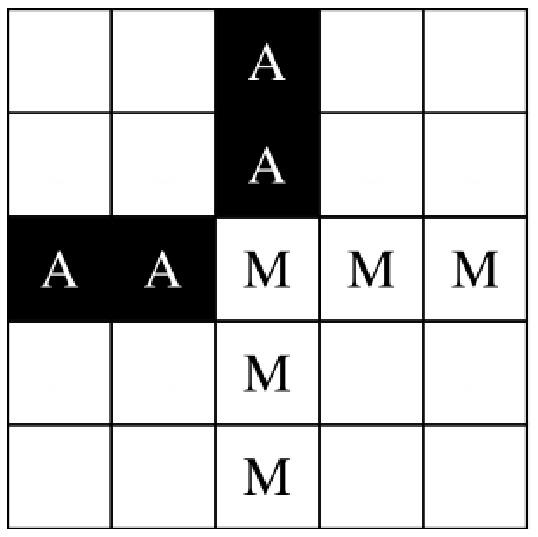}
			
			\caption{Neither of the top two rows and neither of the left two columns can contain AAMMM. }
			\label{fig:thm3-3}
		\end{figure}

	\end{proof}
	By analyzing this argument carefully, it is not difficult to characterize all extremal examples for $\mr{A}^k\mr{M}^{n-k}$.   In particular, Figure \ref{fig:all-3x3} contains all of the extremal constructions up to rotation and reflection of the grid.  Thus, for the advertising job, one can freely choose whichever design one finds to be the most aesthetically pleasing.  We personally recommend one of the middle grids, as these also contain a copy of the word MAA.

	\section{Higher Dimensions.}\label{S-d}
	\subsection{Definitions and upper bounds.}

	Given a set $S$ and positive integer $d$, let $S^d$ denote the set of sequences of length $d$ where every entry is an element of $S$.   Let $\vec{0}$ denote the sequence of $d$ zeros, and if $x,y\in \Z^d$, define $(x+y)\in \Z^d$ by $(x+y)_i=x_i+y_i$. 
	
	Given $p\in [n]^d$ and $v\in \{-1,0,+1\}^d$ with $v\ne \vec{0}$, we say that $\ell:=\{p,p+v,\ldots,p+(n-1)v\}$ is a \textit{line} if each of these elements is in $[n]^d$, and in this case we say that $\ell$ has initial point $p$ and is in direction $v$.  We first make a basic observation that we use throughout this section.
	\begin{lem}\label{L-Boundary}
		Let $\ell$ be a line with initial point $p$ in direction $v$.  Then $v_j=+1$ implies $p_j=1$ and $v_j=-1$ implies $p_j=n$.
	\end{lem}
	\begin{proof}
		Because every point of $\ell$ is in $[n]^d$, in particular we have $p_j,p_j+(n-1)v_j\in [n]$.  If $v_j=+1$ this is only possible if $p_j=1$, and similarly $v_j=-1$ implies $p_j=n$.
	\end{proof}
	
	The line $\ell$ with initial point $p$ in direction $v$ is the same as the line with initial point $p+(n-1)v$ and direction $-v$.  To have a unique identification\footnote{This only makes sense for $n\ge 2,$ as for $n=1$ there is only one line, which is a point, so the direction is not well-defined.} for each line, we say that $(p;v)$ is the \emph{canonical pair} of the line $\ell$ if $\ell$ is the line with initial point $p$ in direction $v$, and if the first nonzero coordinate of $v$ is $+1$.  We also say that $p,v$ are the \emph{canonical initial point} and the \emph{canonical direction} of this line.  
	
	If $\ell$ is a line with canonical pair $(p;v)$, then we let $\ell_i=p+(i-1)v$. We define the \textit{weight} of a line $\ell$ to be the number of nonzero entries in its canonical direction. Note that there are no lines of weight 0 since we always require $v\ne \vec{0}$.
	
	For example, when $d=2$, the $2n$ rows and columns are lines of weight 1 with canonical pairs of the form $((i,1);(0,1))$ and  $((1,i);(1,0))$, and the two diagonals are lines of weight 2 with canonical pairs $((1,1);(1,1))$ and $((1,n),(1,-1))$.  In particular, there are $2n+2$ lines when $d=2$.  The following lemma generalizes these formulas to higher dimensions.
	\begin{lem}\label{L-Tot}
		Let $L_r$ denote the set of lines in $[n]^d$ of weight $r$ and let $L=\bigcup_{r\ge 1} L_r$.  Then for all $r\ge1$,
		\[|L_r|={d\choose r}2^{r-1}n^{d-r},\]
		\[|L|=\half((n+2)^d-n^d).\]
	\end{lem}
	We note that the formula for $L$ in Lemma~\ref{L-Tot} is known, see, e.g., \cite{B, O}, but we include its short proof for completeness. Other results concerning combinatorial properties of lines in higher dimensions can be found in \cite{B, BPV, F}, for example.
	\begin{proof}
		Informally, we wish to define a map $\phi$ from lines $\ell$ to sequences of length $d$ indicating how the coordinates of elements in $\ell$ change as we move along the canonical direction $v$.  For example, if $d=3$ and $\ell=\{(1,1,n),(2,1,n-1),\ldots,(n,1,1)\}$, we define $\phi(\ell)=(+,1,-)$ to indicate that the first coordinate increases as we move along $v$, the second stays fixed at 1, and the third coordinate decreases. 
		
		To make this precise, let $S=\{1,2,\ldots,n,+,-\}^d$.  We define a map $\phi:L\to S$ as follows. If $\ell$ has canonical pair $(p;v)$, we define $\phi(\ell)_i=p_i$ if $v_i=0$, $\phi(\ell)_i=+$ if $v_i=+1$, and $\phi(\ell)_i=-$ if $v_i=-1$.  It is not difficult to see that $\phi$ is an injective map, so it will be enough to determine the cardinality of its image when restricted to each set $L_r$.  
		
		Define $T_r\sub S$ to be the set of sequences with exactly $r$ of its positions equal to $\pm$ and such that the first $\pm$ symbol to appear is a $+$.  Because the canonical direction of a line in $L_r$ has exactly $r$ nonzero coordinates with its first nonzero coordinate equal to $+1$, we see that $\phi(L_r)\sub T_r$, and it is not difficult to see that every element of $T_r$ is mapped to be a unique element of $\phi(L_r)$. Note that  $|T_r|={d\choose r}2^{r-1}n^{d-r}$, as every element of $T_r$ can be identified by choosing the $r$ positions that are $\pm$, then choosing for all but the first of these positions which of the two  symbols $\pm$ to take on, and then choosing any element in $[n]$ for each of the remaining positions.  This proves the first result.
		
		The second result follows from the binomial theorem and summing the bound for each $L_r$.  More directly, $\phi(L)$ is the set of sequences with at least one coordinate equal to $\pm$ with the first of these being a $+$.  There are $(n+2)^d-n^d$ sequences with at least one $\pm$ symbol, and exactly half of these sequences have the first such symbol as $+$.
	\end{proof}
	
	We say that $G$ is an $(n,d)$-\emph{grid} (or simply a ``grid'' if $n$ and $d$ are understood) if it is a function from $[n]^d$ to a set of letters.  Given a word $w$ of length $n$ and a grid $G$, we say that a line $\ell$ contains $w$ if either $G(\ell_1)\cdots G(\ell_n)=w$ or $G(\ell_n)\cdots G(\ell_1)=w$.  We let $f(w,G)$ denote the number of lines containing $w$ in $G$, and we define \[f(w,d)=\max_{G\tr{ an }(n,d)\tr{-grid}} f(w,G).\]  Observe that Lemma~\ref{L-Tot} immediately gives that $f(w,d)\le \half((n+2)^d-n^d)$ for all $w$ and $d$.  We can improve upon this bound for nonsymmetric words with a generalization of Lemma~\ref{L-SharpTot2}.
	
	\begin{lem}\label{L-SharpTot}
		Let $w$ be a word with $w_i\ne w_{n-i+1}$ for some $i$.  Then \[f(w,d)\le \quart((n+2)^d-(n-2)^d).\]
	\end{lem}
	\begin{proof}
		Without loss of generality we assume that $i < n-i+1.$ The proof will proceed in four steps:
		\begin{enumerate}
			\item[1:] We define sets $Q_p$ for certain $p\in[n]^d$. These $Q_p$ correspond to corners of sub-hypercubes in $[n]^d,$ and the $p$ are the ``upper-left" corners of these sub-hypercubes. The $Q_p$ sets will partition $[n]^d.$ 
			\item[2:] We classify every line in $[n]^d$ according to which $Q_p$ it intersects. In particular, we show that for each line there is exactly one $p$ such that the line intersects $Q_p$ twice.
			\item[3:] We determine an upper bound for the number of lines that can intersect each $Q_p$ in two places and that can contain $w$.
			\item[4:] We sum over the upper  bounds in Step 3.
		\end{enumerate}

		
		\textbf{Step 1:} Let $P$ be the set of points $p\in [n]^d$ that do not have any position equal to $n-i+1$.  Let $I(p)=\{j:p_j=i\}$, and for each $p\in P$, define $Q_p$ to be the set of points $p'\in [n]^d$ with $p'_j=p_j$ if $j\notin I(p)$ and $p'_j\in \{i,n-i+1\}$ otherwise. In other words, $Q_p$ is the set of corners of an $|I(p)|$-dimensional sub-hypercube that has $p$ as one of its corners. 
		
		Note that each $p'\in [n]^d$ is in exactly one $Q_p$ set, namely for the $p$ that has $p_j=i$ whenever $p'_j=n-i+1$ and that has $p_j=p'_j$ otherwise.  Also observe that $|Q_p|=2^{|I(p)|}$, since each $p'\in Q_p$ has $p'_j\in \{i,n-i+1\}$ for each $j\in I(p)$ and this completely determines $p'$. In particular, $|Q_p|=1$ for all $p\in P$ such that $p$ has no coordinate equal to $i.$
		
		\textbf{Step 2:} Let $L$ be the set of lines in $[n]^d$. Given an $\ell\in L$ with canonical pair $(p';v)$, define \[\phi(\ell)=\{p'+(i-1)v,p'+(n-i)v\};\] we call these two points the \textit{$i$-dentifiers} of $\ell$. (For example, the 1-dentifiers of $\ell$ are the endpoints of $\ell$.)  We claim that $\phi$ is a bijection from $L$ to $\bigcup_{p\in P}{Q_p\choose 2}$, where ${Q_p\choose 2}$ denotes the set of two-element subsets of $Q_p$. 
		
		Let $\ell$ be a line with canonical pair $(p';v)$, and define $p\in [n]^d$ by 
		\[p_j=\begin{cases}
			p'_j & v_j=0\tr{ and }p'_j\ne n-i+1,\\ 
			i & \tr{otherwise}.
		\end{cases}\]  We have $p\in P$ by construction.
		
		If $p_j\neq i$, then by construction we have that $p_j = p_j'$ and $v_j = 0.$ Therefore $p_j'+(i-1)v_j = p_j.$ On the other hand, suppose $p_j = i.$ Then either $v_j \neq 0,$ or $v_j = 0$ and $p_j' \in \{i,n-i+1\}$. In the former case, Lemma~\ref{L-Boundary} forces $p_j'\in\{1,n\}$ and therefore $p_j' + (i-1)v_j \in\{i,n-i+1\}$. In the latter case, it is clear that $p_j' + (i-1)v_j  = p_j'\in \{i,n-i+1\}.$ By definition, this means $p' + (i-1)v \in Q_p.$ Similarly, $p' + (n-i)v \in Q_p.$ Thus $\phi(\ell) \in {Q_p\choose2}$ and $\phi$ maps to the desired codomain.
		
		
		The map $\phi$ is injective since any two points of $\ell$ determine the line. It remains to prove that $\phi$ is surjective.  Fix some $p\in P$ and $\{q,q'\}\in {Q_p\choose 2}$.  Observe that if $q_j\ne q'_j$ for some $j$, then $|q_j-q'_j|=n-2i+1$.  Since $q\ne q'$, the sequence $v=\rec{n-2i+1}(q-q')$ is a nonzero element of $\{-1,0,1\}^d$, and possibly by relabeling $q,q'$, we can assume that the first nonzero coordinate of $v$ is $+1$.  Define the point $p':=q-(n-i)v$. Note that when $v_j = +1,$ we have $q_j = n-i+1$, so $p_j' = 1.$ Similarly, when $v_j = -1$ we have $q_j = i$ and $p_j' = n.$  
		Thus $p'+(j-1)v\in [n]^d$ for all $1\le j\le n$ and we have that $(p';v)$ is the canonical pair of some line $\ell$ with $\phi(\ell)=\{q,q'\}$, proving the surjectivity of the map.
		
		\textbf{Step 3:} Let $L_p$ be the set of lines whose $i$-dentifiers are in $Q_p$, i.e., $L_p$ is the preimage of ${Q_p\choose 2}$ under $\phi$.   Let $f(w,L_p)$ be the number of lines of $L_p$ that contain $w$ in $G$.  Note that $\phi$ being a bijection implies that each line belongs to exactly one $L_p$ set, so
		\[f(w,G)=\sum_p f(w,L_p).\]
		Given $p\in P$, let $Q_p^+=\{q\in Q_p:G(q)=w_i\}$ and $Q_p^-=\{q\in Q_p:G(q)=w_{n-i+1}\}$.  Because $w_i\ne w_{n-i+1}$ by assumption, for a line $\ell \in L_p$ to contain $w$, we must have one of its $i$-dentifiers be in $Q_p^+$ and the other in $Q_p^-$.  Because each $\{q,q'\}\in {Q_p\choose 2}$ is the $i$-dentifier of a unique line in $L_p$, we have \begin{align*}f(w,L_p)&\le \l|\l\{\{q,q'\}\in {Q_p\choose 2}:q\in Q_p^+,\ q'\in Q_p^-\r\}\r|\\ &= |Q_p^+||Q_p^-|\le \l(\half |Q_p|\r)^2=4^{|I(p)|-1},\end{align*}and in particular this integer is $0$ when $|I(p)|=0$.  
		
		\textbf{Step 4:} Let $P_r\sub P$ be the set of $p$ with $|I(p)|=r$.  Note $|P_r|={d\choose r}(n-2)^{d-r}$, as a point in $P_r$ is uniquely determined by first choosing which of its $r$ coordinates are equal to $i$, and then each of the remaining coordinates can be any element in $[n]\sm \{i,n-i+1\}$.  We conclude that
		{\small \begin{align*}
				f(w,G)=\sum_{r=0}^d \sum_{p\in P_r} f(w,L_p)&\le \sum_{r=1}^d {d\choose r}(n-2)^{d-r} \cdot 4^{r-1}=\quart ((n+2)^d-(n-2)^d),
		\end{align*}}
		where the last step used the binomial theorem.
	\end{proof}
	
	\subsection{Lower bounds.}
	We recall a standard combinatorial lemma.
	
	\begin{lem}\label{L-Odd}
		For $d$ a positive integer and $x,y\in \R$, we have \[\sum_{r\tr{ odd}} {d\choose r} x^ry^{d-r}=\half((x+y)^d-(-x+y)^d).\]
	\end{lem}
	\begin{proof}
		Using the binomial theorem, we see that the right-hand side of this equation is equal to \[\half\sum_r {d\choose r}\l(x^ry^{d-r}-(-x)^ry^{d-r}\r)=\sum_{r\tr{ odd}} {d\choose r} x^ry^{d-r}.\]
	\end{proof}
	With this we can prove our exact result for anti-symmetric words in all dimensions.
	\begin{proof}[Proof of Theorem~\ref{T-AntiSym}]
		The upper bound follows from Lemma~\ref{L-SharpTot}.  For simplicity we denote the two letters of $w$ by $\pm 1$ and we assume $w_1=+1$.  Let $I$ denote the set of $i$ such that $w_i=+1$, so that $w_i=-1$ for all $i\notin I$.  Given a point $p\in [n]^d$, define $\sig_I(p)=|\{i:p_i\in I\}|$, and define the grid $G$ by $G(p)=(-1)^{\sig_I(p)}$.
		
		We claim that every line of $G$ of odd weight contains $w$.  Indeed, let $\ell$ be a line of odd weight $r$ with canonical pair $(p;v)$, and recall that $\ell_i:=p+(i-1)v$.  Let us first assume $\sig_I(\ell_1)$ is even, so that $G(\ell_1)=+1=w_1$.  Inductively assume $G(\ell_i)=w_i$.  Note that as we go from $\ell_i$ to $\ell_{i+1}$, a total of $r$ coordinates will change, each going from either $i$ to $i+1$ or from $n-i+1$ to $n-i$, with all other coordinates remaining the same.  If $w_{i+1}=w_i$, then we either have $i,i+1\in I$ or $n-i+1,n-i\in I$, and in either case we have $\sig_I(\ell_{i+1})=\sig_I(\ell_i)$ and $G(\ell_{i+1})=w_i=w_{i+1}$, as desired.  If instead $w_{i+1}\ne w_i$, then either $i,n-i+1\in I$ or $i+1,n-i\in I$.  Depending on the case, we have $\sig_I(\ell_{i+1})=\sig_I(\ell_i)\pm r$.  Thus $G(p)=w_i (-1)^r=-w_i=w_{i+1}$, where we used the assumption that $r$ has odd weight.  Essentially the same proof works when $\sig_I(\ell_1)$ is odd, but with our reading the words backwards instead of forwards.  By  Lemmas~\ref{L-Tot} and \ref{L-Odd} we have
		\[f(w)\ge f(w,G)\ge \half \sum_{r\tr{ odd}}{d\choose r}2^r n^{d-r}=\quart((n+2)^d-(n-2)^d).\]
	\end{proof}
	
	The idea of the proof of Theorem~\ref{T-HighD} is as follows.  Define \[\pi_i(p):=|\{j:p_j=i\}|,\hspace{3em}\tau_i(p):=\pi_i(p)+\pi_{n-i+1}(p).\]  By looking at the problem through a probabilistic lens, one can show that with high probability a random line of $[n]^d$ has weight roughly $\f{2d}{n+2}$.  Moreover, for any point $p$ on this line, with high probability there is a unique $i\le \ceil{n/2}$ such that $\tau_i(p)\approx \f{4d}{n+2}$ with $\pi_j(p)\approx \f{d}{n+2}$ for all $j\ne i,n-i+1$.
	Motivated by this, if $w$ is symmetric, we will define our grid $G$ by setting $G(p)=w_i$ whenever there is a unique $i\le \ceil{n/2}$ with $\tau_i(p)$ large.  This will give us almost every line for symmetric words.  
	
	To get the bound for nonsymmetric words, we further need to decide whether a point $p$ with  $\tau_i(p)\approx \f{4d}{n+2}$ should be assigned to $w_i$ or $w_{n-i+1}$.  We will make this decision based on the parity of \[\sig(p):=\sum_{i\le n/2} \pi_i(p).\]  Note that this is a special case of the function $\sig_I(p)$ used in the proof of Theorem~\ref{T-AntiSym} with $I=[n/2]$ (and in fact, any such $\sig_I$ can be used to give the desired result).  The reason we make this definition is because of the following.
	
	\begin{lem}\label{L-Side}
		Let $(p;v)$ be the canonical pair of a line $\ell$ of odd weight.  Then
		\begin{align*}
			\sig(p)\equiv \sig(p+(i-1)v)\mod 2\hspace{3em} i\le \floor{n/2},\\ \sig(p)\not\equiv \sig(p+(i-1)v)\mod 2\hspace{3em}i>\ceil{n/2}.
		\end{align*}
	\end{lem}
	Heuristically this says that we can use the parity of $\sig(p)$ to determine whether we are on the ``left'' or ``right'' side of our desired line.  We note that for $n$ odd, this lemma says nothing about $i=(n+1)/2$.  This is fine because we can arbitrarily assign this point to be on the left or right side of the line since $w_i=w_{n-i+1}$ in this case. 
	\begin{proof}
		Let $\pi'_j(p)=|\{k:p_k=j,\ v_k=0\}|$.  Thus for all $i$, $\pi_j(p+(i-1)v)$ will always be at least $\pi'_j(p)$ (since these positions stay fixed for all $i$) plus any of the positions that are made equal to $j$ due to moving along $v$.  By Lemma~\ref{L-Boundary}, if $v_k=+1$ then $(p+(i-1)v)_k=i$, and if $v_k=-1$ then $(p+(i-1)v)_k=n-i+1$.  Thus if $r$ is the weight of $\ell$ and $s$ is the number of positions $k$ with $v_k=+1$, we have
		\begin{align*}\pi_j(p+(i-1)v)&=\pi'_j(p)\hspace{1em}\tr{if }j\ne i,n-i+1,\\ \pi_i(p+(i-1)v)&=\pi'_i(p)+s,\\ \pi_{n-i+1}(p+(i-1)v)&=\pi'_{n-i+1}(p)+r-s.\end{align*} 
		In particular, every term of $\sig(p+(i-1)v)=\sum_{j\le n/2} \pi_j(p+(i-1)v)$ will equal $\pi'_j(p)$ except for $j=\min\{i,n-i+1\}$, at which point it is equal to either $\pi'_j(p)+s$ or $\pi'_j(p)+r-s$ depending on which value $j$ takes.  Thus for $i\le \floor{n/2}$ (and in particular for $i=1$), we have 
		\[\sig(p+(i-1)v)=s+\sum_{j\le n/2} \pi'_j(p)=\sig(p),\]
		and this second equality shows that, for $i>\ceil{n/2}$, we have 
		\[\sig(p+(i-1)v)=r-s+\sum_{j\le n/2} \pi'_j(p)=r-2s+\sig(p).\]
		Because $r$ is odd, this last quantity is of different parity from $\sig(p)$.
	\end{proof}

	We now begin the setup for the formal proof of Theorem~\ref{T-HighD}.  Let $w$ be a word of length $n\ge 3$ and define  \[c:=\f{3.9d}{n+2},\hspace{1em} k:=\f{2.3d}{n+2}.\]   We say that a point $p$ is a \emph{counter-point} if \begin{align*}c&<\tau_1(p)< 1.1c \hspace{7em} \text{and}\\   \f{.99(d-\tau_1(p))}{n-2}&\le \pi_i(p)\le \f{1.01(d-\tau_1(p))}{n-2} \quad \text{for all }1<i<n.\end{align*}  We define the grid $G_{w}$ on points $p$ with $\sig(p)$ odd as follows. $G_{w}(p)=w_1$ if $p$ is a counter-point, $G_{w}(p)=w_i$ if $\tau_1(p)\le c$ and if $\tau_i(p)\ge k$ for a unique $1<i\le n/2$, and $G_{w}(p)$ is assigned arbitrarily for all other points with $\sig(p)$ odd.  If $\sig(p)$ is even we let $G_{w}(p)=w_n$ if $p$ is a counter-point, $G_{w}(p)=w_{n-i+1}$ if $\tau_1(p)\le c$ and $\tau_i(p)\ge k$ for a unique $1<i\le n/2$, and $G_{w}(p)$ is assigned arbitrarily for all other points.
	
	Our first goal is to show that most lines whose endpoints are counter-points contain the word $w$.  To do this, we use a version of the Chernoff bound which can be found in \cite{AS}, for example.
	
	\begin{lem}[Chernoff bound]\label{L-Chern}
		Let $\Bin(n,p)$ denote a binomial random variable with $n$ trials and probability $p$, i.e., $\Pr[\Bin(n,p)=k]={n\choose k}p^k(1-p)^{n-k}$.  Then for all $\lam>0$ we have \[\Pr[|\Bin(n,p)-pn|>\lam pn]\le 2e^{-\lam^2 pn/2}.\]
	\end{lem}
	That is, with high probability binomial variables are close to their expectation.
	
	\begin{lem}\label{L-many-lines}
		Let $w$ be a word of length $n\ge 3$ and $p$ a counter-point with $\tau_1(p)=r$. Then the number of lines $\ell$ in $G_w$ that have initial point $p$ and that contain $w$ is at least $(1-2e^{-c/10000})2^{r-1}$.  If $w_i=w_{n-i+1}$ for all $i$, then the number of lines is at least $(1-2e^{-c/10000})2^{r}$.
	\end{lem}
	Note that in the statement of the lemma we do not require $p$ to be a canonical initial point.
	\begin{proof}
		Let $S\sub [d]$ be a set of coordinates with $|S|=s\ge .49 r$ such that $j\in S$ implies $p_j\in \{1,n\}$.  Note that there are exactly ${r\choose s}$ ways to choose such a set $S$ with $|S|=s$ by the hypothesis $\tau_1(p)=r$.  Given such an $S$, we form a line $\ell_S$ with initial point $p$ and direction $v$ defined by $v_j=0$ if $j\notin S$, and otherwise $v_j=+1$ if $p_j=1$ and $v_j=-1$ if $p_j=n$.  
		
		For any $i$ and $j\ne 1,n,i$,  we have \begin{align*}\tau_j(p+(i-1)v)&=\tau_j(p)\le 2.02\f{d-r}{n-2}< 2.02\f{d-c}{n-2}\\ &=\f{2.02d}{n+2}\cdot \f{n-1.9}{n-2}\le \f{2.222d}{n+2}< k,\end{align*} where the first equality used similar reasoning as in the proof of Lemma~\ref{L-Side}, the first two inequalities used the definition of counter-points and $j\ne 1,n$,  and the penultimate inequality used the fact that $n\ge 3$.  We also have for $i\ne 1,n$ that \[\tau_i(p+(i-1)v)=\tau_i(p)+s\ge 1.98\f{d-r}{n-2}+.49r.\]
		Because $c < r < 1.1c$, this implies
		\begin{align*}
			\tau_i(p+(i-1)v)&> \f{1.98(d-1.1c) + 0.49c(n-2)}{n-2}\\&=\f{d}{n+2}\f{3.891n-8.3562}{n-2}\\
			&\ge \f{3.3168 d}{n+2}> k.
		\end{align*}
		Finally, because $r< 1.1c$, we have for $i\ne 1,n$, \[\tau_1(p+(i-1)v)=r-s\le 0.51 r< c.\]  Thus $G_w(p+(i-1)v)\in \{w_i,w_{n-i+1}\}$ for all $i$  (this holds for $i=1,n$ because $p$ is a counter-point), with the exact value depending on the parity of $\sig(p+(i-1)v)$.  If $w_i=w_{n-i+1}$ for all $i$, then $\ell_S$ contains $w$.  If $w$ is not symmetric and $\ell_S$ is of odd weight, then by Lemma~\ref{L-Side} we will either have $\sig(p+(i-1)v)=w_i$ for all $i$ if $\sig(p)$ is odd, or $\sig(p+(i-1)v)=w_{n-i+1}$ for all $i$ otherwise.  In either case $\ell_S$ contains the word, so we conclude that $\ell_S$ always contains $w$ if it has odd weight.
		
		It remains to count how many lines contain $w$ in each of these cases.  For the nonsymmetric case, each set $S$ of odd size $s\ge .49 r$ gives a distinct line $\ell_S$ containing $w$.  Using Lemmas~\ref{L-Odd} and \ref{L-Chern}, we find that the number of such lines is 
		\begin{align*}
			\sum_{s\ge .49r,\ s\tr{ odd}} {r\choose s} &= 2^{r-1} - \sum_{s< .49r,\ s\tr{ odd}} {r\choose s} \\&\ge 2^{r-1}-2^r\sum_{s<.49r}{r\choose s} 2^{-r}\\
			&=2^{r-1}-2^r\cdot \Pr[\Bin(r,.5)<.49r]\\
			&= 2^{r-1}-2^{r-1}\cdot \Pr[|\Bin(r,.5)-.5r| > .01r]\\
			&\ge (1-2e^{-r/10000})2^{r-1}\ge (1-2e^{-c/10000})2^{r-1}.
		\end{align*}
		The analysis for the symmetric case is essentially identical, the only difference being that we sum over all $s\ge .49r$ regardless of its parity.
	\end{proof}
	
	We next wish to show that there are many counter-points.
	
	\begin{lem}\label{L-many-good}
		Let $w$ be a word of length $n\ge3$. For $c<r<1.1c,$ the number of counter-points with $\tau_1(p)=r$ is at least \[(1-2de^{-d/40000(n+2)}){d\choose r} 2^r(n-2)^{d-r}.\]
	\end{lem}
	\begin{proof}
		Uniformly at random choose a point with $\tau_1(p)=r$ by first uniformly choosing a set $J\subset[d]$ of size $r$ and uniformly assigning $p_j\in \{1,n\}$ for all $j\in J$, and then uniformly choosing $p_j\in \{2,\ldots,n-1\}$ for all other $j$. 
		
		Note that for $j\neq1,n$ the distribution of $\pi_j(p)$ is exactly $\Bin(d-r,\f{1}{n-2})$, and thus the probability that any $\pi_j(p)$ is outside the range $\f{.99(d-r)}{n-2}$ and $\f{1.01(d-r)}{n-2}$ is at most $2e^{-(d-r)/
			20000(n-2)}$ by Lemma~\ref{L-Chern}.  Because $r<1.1c$ and $n\geq3$, 
		\begin{align*}
			\f{d-r}{n-2}&> \f{d-1.1\f{3.9}{n+2}d}{n-2} = \f{d(n+2) - 4.29d}{(n-2)(n+2)}=\f{d(n-2.29)}{(n+2)(n-2)}\\
			&\ge 0.71\cdot \f{d}{n+2}\ge 0.5 \cdot \f{d}{n+2}.
		\end{align*}
		Therefore the probability that any $\pi_j(p)$ is outside the range $\f{.99(d-r)}{n-2}$ and $\f{1.01(d-r)}{n-2}$ is at most $2e^{-(d-r)/
			20000(n-2)}\le  2e^{-d/40000(n+2)}.$	By the union bound, the probability that none of the $\pi_j(p)$ are outside this range, and hence the probability that $p$ is a counter-point, is at least $1-2d e^{-d/40000(n+2)}$.
		
		
		Because we chose $p$ uniformly among all points with $\tau_1(p)=r$, and because the number of such points is exactly ${d\choose r} 2^r (n-2)^{d-r}$, we conclude that \[\f{|\{p:\tau_1(p)=r,\ p\tr{ counter-point}\}|}{{d\choose r} 2^r(n-2)^{d-r}}\ge 1-2de^{-d/40000(n+2)},\] which gives the desired result.
	\end{proof}
	
	We are now ready to prove Theorem~\ref{T-HighD}.
	\begin{proof}[Proof of Theorem~\ref{T-HighD}]
		For $n=2$, one can either have a word of the form $w=\mr{AA}$ (in which case the grid that always maps to A has every line containing $w$), or of the form $w=\mr{AM}$ (in which case Theorem~\ref{T-AntiSym} applies and gives the desired bound).  From now on we assume $n\ge3$. 
		
		Consider $G_w$ as defined above.  By using Lemmas~\ref{L-many-lines} and \ref{L-many-good} (and noting that the former lemma double counts lines since we do not consider canonical initial points), we find 
		
		\[f(w,G)\hspace{-0.18pt}\ge \half \sum_{c< r< 1.1c} (1-2de^{-d/40000(n+2)}){d\choose r}2^r(n-2)^{d-r} (1-2e^{c/10000}) 2^{r-1}\]
		{\small \[= (1-2de^{-d/40000(n+2)})(1-2e^{c/10000})\quart (n+2)^d\sum_{c< r< 1.1c} {d\choose r}\left(\f{4}{n+2}\right)^r\left(\f{n-2}{n+2}\right)^{d-r}\]}
		Note that this sum is equal to \[\Pr\l[c< \Bin\l(d,\f{4}{n+2}\r)< 1.1c\r]\ge \Pr\l[\f{3.9d}{n+2}< \Bin\l(d,\f{4}{n+2}\r)< \f{4.1 d}{n+2}\r],\] where we used $4.1 <  3.9\cdot 1.1$. By Lemma \ref{L-Chern}, this is greater than $1 - 2e^{-d/800(n+2)}$. Hence
		\[f(w,G) \ge(1-2de^{-d/40000(n+2)})(1-2e^{c/10000})(1 - 2e^{-d/800(n+2)})\quart (n+2)^d. \]
		
		Using this lower bound (and the fact that $c\to \infty$ as $d\to \infty$), as well as Lemma \ref{L-SharpTot}, we conclude for nonsymmetric words that $f(w,G) \sim \quart (n+2)^d$. For symmetric words, we use the symmetric case of Lemma \ref{L-many-lines} and an analogous argument to conclude the desired result.
	\end{proof}
	
	\section{A few short words on short words.}\label{sec:k}
	
	In this section, we informally discuss results for words of length $k$ in $(n,d)$-grids.  We define $\ell$ to be a line of length $k$ in the obvious way, and, similarly for an $(n,d)$-grid $G$, we define what it means for this line to contain a word $w$ of length $k$ in a manner analogous to what we did before.  We let $f(w,n,d)=\max_G f(w,G)$.  One can show that the total number of lines in $[n]^d$ of length $k$ is exactly \[\half((3n-2k+2)^d-n^d).\]
	The proof is essentially the same as that of Lemma \ref{L-Tot}.  The main difference is that before if we had $\phi(\ell)_i=+$ we knew that $p_i=1$, but now it could be any value that is at most $n-k+1$.  Thus, instead of writing $+$, we should use the symbol $p_i^+$ to specify this information, and similarly one should use $p_i^-$ instead of $-$.
	
	Note that if $k$ is much smaller than $n$ and $d$ is fixed, then asymptotically this number is $\f{3^d-1}{2} n^d$.  This can be seen more directly by choosing a random point $p$ in $[n]^d$ and observing that with high probability, $p$ is one of the two endpoints of a line in direction $v$ for all of the $3^d-1$ possible nonzero directions $v$.
	
	For the rest of this section we focus primarily on the case $d=2$, with many of these ideas carrying over to larger $d$.  The best generic bounds we have in this case are the following.
	
	\begin{thm}
		\label{T-sml-wrd}
		If $w$ is a word of length $k$, then \[f(w,n,1)\cdot (3n-4k)\le f(w,n,2)\le f(w,n,1)\cdot 2n+4\sum_{i=k}^n f(w,i,1).\]
	\end{thm}
	Note that if $f(w,n,1)\sim \al n$ and $k$ is much smaller than $n$, then asymptotically the lower bound is $3\al n^2$ and the upper bound is $4\al n^2$, so this essentially solves the problem within a factor of 4/3 for all words $w$.
	\begin{proof}
		Let $G$ be an $(n,2)$-grid.  By definition, the number of copies of $w$ that $G$ contains in its first row is at most $f(w,n,1)$, so in total the total number of copies it contains in its rows is at most $f(w,n,1)\cdot n$.  The same result holds for the number of copies of $w$ appearing in a column.  Similarly there are (at most) four ``maximal'' diagonal lines in $G$ of length $i$ for any $i\ge k$, see Figure~\ref{fig:diagonal}, and within each of these diagonals, $G$ contains at most $f(w,i,1)$ copies of $w$.  As every line of length $k$ appears in a row, column, or maximal diagonal of length at least $k$, we conclude the upper bound.
		\begin{figure}[h]
			\centering
			\includegraphics[width=.6\textwidth]{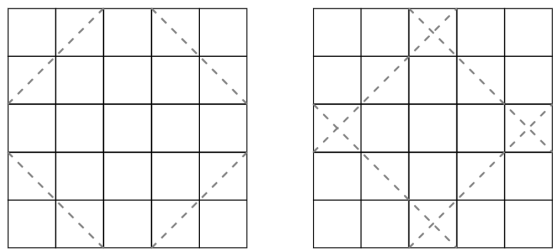}
			
			\caption{Dashed lines showing the four maximal diagonals of lengths 2 and 3 in a $5\times5$ grid.}\label{fig:diagonal}
			
		\end{figure}
		
		For the lower bound, let $G'$ be an $(n,1)$-grid that contains $f(w,n,1)$ copies of $w$.  Define $G(i,j)=G'(j)$.  Observe then that the total number of copies of $w$ appearing in rows is exactly $f(w,n,1)\cdot n$.  Moreover, if $v=(0,1)$ and $p=(i,j)$ with $k\le i\le n-k$ is such that the line with initial point $p$ and direction $v$ contains $w$ in $G$, then so do the lines with initial points $p$ in directions $v'=(1,1)$ and $v''=(-1,1)$.  This gives an extra count of $f(w,n,1)\cdot 2(n-2k)$, as desired.
	\end{proof}
	For example, if $w'$ is the word of length $k\ge 2$ with $k$ distinct letters, we claim that $f(w',n,1)\approx n/(k-1)$.  For $w'=\mr{ABCD}$, the corresponding lower bound construction looks like Figure \ref{fig:1-d}:

	\begin{figure}[h]
		\centering
		\includegraphics[width=.9\textwidth]{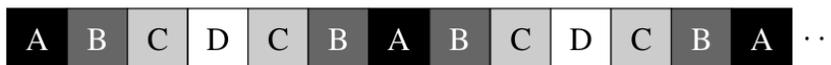}
		
		\caption{A one-dimensional grid giving the maximum number of instances of the word ABCD.}\label{fig:1-d}
		
	\end{figure}
	
	In this grid, almost every copy of $\mr{A}$ and $\mr{D}$ is one of the two endpoints of two lines containing $w'$, so the total number of lines containing $w'$ is roughly the number of times $\mr{A}$ and $\mr{D}$ appear in this grid.  It is not hard to see that this happens for roughly one third of the positions in $G$, and a similar construction proves the lower bound of our claim for all $k$.  For the upper bound, we have $f(w',k,1)=1$, and it is not difficult to see that for all $n$, \[f(w',n+k-1,1)\le f(w',n,1)+1,\] as the first $n$ letters of an $(n+k-1,1)$-grid contain at most $f(w',n,1)$ copies of $w'$ and at most one of the last $k-1$ positions can be the endpoint of at most one line containing $w'$.
	
	This claim together with Theorem \ref{T-sml-wrd} gives that $f(w',n,2)$ is roughly between $\f{3}{k-1}n^2$ and $\f{4}{k-1}n^2$ when $w'$ is a word on $k$ distinct letters, and in general we do not know what the correct answer is for this very simple family of words, though we suspect that the lower bound is closer to the truth.
	
	\begin{conj}
		If $w'$ is the word of length $k\ge 2$ consisting of distinct letters, then \[f(w',n,2)\sim \f{3}{k-1}n^2.\]
	\end{conj}
	In fact, it may be that $f(w,n,2)\sim f(w,n,1)\cdot 3n$ for any fixed word $w$.
	\section{Appendix.}\label{Sec-App}
	
	In order to get a better feel for the problem in higher dimensions, Figure~\ref{fig:my_label} shows one of the three (up to rotation and reflection) optimal AMM grids for $d=3$.  More precisely it displays an ``unfolded'' two-dimensional version of the grid.  To recover the three-dimensional grid, one can cut out the diagram and fold it into a cube along the gray lines (where the center of the cube is assumed to be M).  There are 28 instances of AMM in this cube --- can you see them all?  
	
	\begin{figure}[h]
		\centering
		\includegraphics[width=4.5in]{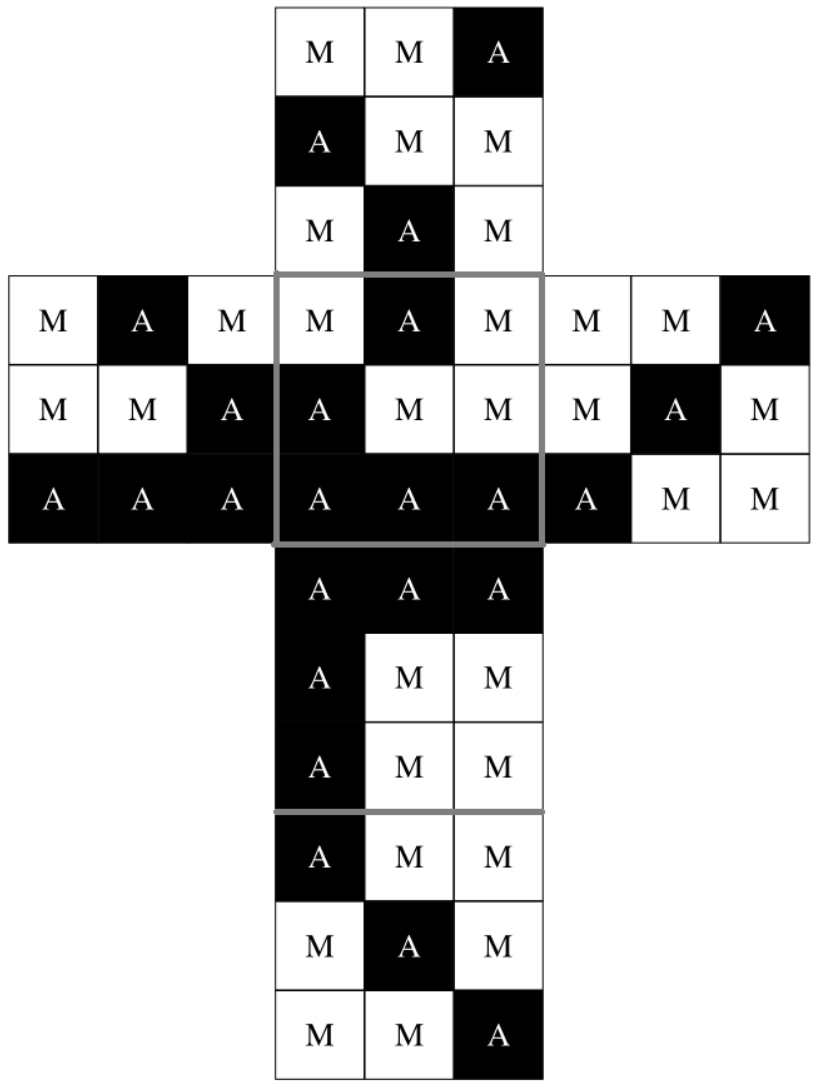}
		
		\caption{An unfolded $3\times3\times3$ grid giving the maximum number of instances of AMM.}
		\label{fig:my_label}
	\end{figure}

	\begin{acknowledgment}{Acknowledgments.}
		The authors are indebted to the two anonymous referees who greatly improved the quality of this paper.  The authors wish to thank M. R. Thought for proposing this research problem and for fruitful discussions.  The second author is supported by the National Science Foundation Graduate Research Fellowship under Grant No. DGE-1650112. 
	\end{acknowledgment}

	\begin{biog}
		\item[Gregory Patchell] is a PhD student at UC San Diego studying operator algebras under Adrian Ioana. He is easily distracted by problems that are simple to state but hard to understand. Outside of math, his interests include drinking IPAs and avoiding barbers.
		
		\begin{affil}
			Department of Mathematics, UC San Diego, La Jolla CA 92093\\
			gpatchel@ucsd.edu
		\end{affil}
		
		\item[Sam Spiro] is a PhD student at UC San Diego studying combinatorics with Jacques Verstraete.  In addition to looking at words in grids, he also enjoys looking at words in books and manga.
		\begin{affil}
			Department of Mathematics, UC San Diego, La Jolla CA 92093\\
			sspiro@ucsd.edu
		\end{affil}
	\end{biog}
	
	\vfill\eject

\begin{thebibliography}{1}
		\bibitem{AS} Alon, N., Spencer, J. (2004). \textit{The Probabilistic Method}, 2nd ed. New York, NY: Wiley--Interscience.
		
		\bibitem{B} Beck, J. (2008). \textit{Combinatorial Games: Tic-Tac-Toe Theory}. New York, NY: Cambridge Univ. Press. 
		
		\bibitem{BPV} Beck, J., Pegden, W., Vijay, S. (2009). The Hales--Jewett number is exponential: game-theoretic consequences. In: Chen, W. W. L., Gowers, W.T., Halberstam, H., Schmidt, W. M., Vaughan, R. C., eds. \textit{Analytic Number Theory: Essays in Honour of Klaus Roth},  Vol.\ 14. Cambridge Univ. Press, pp.\ 22--37.  
		
		\bibitem{DK} D\'enes, J., Keedwell, D. (2015). \textit{Latin Squares and Their Applications}, 2nd ed. Amsterdam: Elsevier.
		
		\bibitem{DM} D\'enes, J., Mullen, G. (1993). Enumeration formulas for Latin and frequency squares. \textit{Discrete Math} 111(1--3): 157--163. doi.org/10.1016/0012-365X(93)90152-J
		
		\bibitem{E} Erd\H{o}s, P., Hickerson, D., Norton, D., Stein, S. (1988). Has every Latin square of order $n$ a partial Latin transversal of size $n-1$? \textit{Amer. Math. Monthly}. 95(5): 428--430. doi.org/10.1080/00029890.1988.11972024
		
		\bibitem{F} Felzenbaum, A., Holzman, R., Kleitman, D. J. (1993). Packing lines in a hypercube. \textit{Discret. Math.} 117(1--3): 107--112. doi.org/10.1016/0012-365X(93)90327-P
		
		\bibitem{MW} McKay, B., Wanless, I. (2008). A census of small Latin hypercubes. \textit{SIAM J. Discret. Math.} 22(2): 719--736. doi.org/10.1137/070693874
		
		\bibitem{O} OEIS Foundation Inc. (2020). The On-Line Encyclopedia of Integer Sequences. oeis.org/A005059
		
		\bibitem{U} Ullrich, P (1999). An Eulerian square before Euler and an experimental design before R.A. Fisher: On the early history of Latin squares. \textit{Chance}. 12(1): 22--26. doi.org/10.1080/09332480.1999.10542137
		
		
		\bibitem{ZK} Zaikin, O.,  Kochemazov, S. (2015). The search for systems of diagonal Latin squares using the SAT@home project. \textit{Int. J. Intell. Inf. Technol}. 3(11): 4--9.
		
		
	\end{thebibliography}
\end{document}